\documentclass[12pt]{article}
\usepackage[T2A]{fontenc}
\usepackage[cp1251]{inputenc}
\usepackage[english]{babel}
\usepackage{amsmath,latexsym,amsthm,amsfonts,tipa}
\usepackage{amssymb,amscd,graphpap}
\newcommand\NN{{\mathbb N}}

\newcommand\w{{\omega}}
\newcommand\BB{{\mathcal B}}
\newcommand\II{{\mathcal I}}
\newcommand\PP{{\mathcal P}}
\newcommand\UU{{\mathcal U}}
\newcommand\VV{{\mathcal V}}
\newtheorem{Th}{Theorem}
\newtheorem{Lm}{Lemma}
\theoremstyle{definition}
\newtheorem{Ex}{Example}
\newtheorem{Rm}{Remark}

\newtheorem{Pb}{Problem}

\begin{document}

\pagestyle{empty}
\title{Asymptotically Scattered Spaces}
\author{Igor Protasov}
\date{}
\maketitle

\begin{abstract}
We define thin and asymptotically scattered metric spaces as asymptotic counterparts of discrete and scattered metric spaces respectively. We characterize asymptotically scattered spaces in terms of prohibited subspaces, and classify thin metric spaces up to coarse equivalence. We introduce the types of asymptotically scattered spaces and construct the spaces of distinct types. We transfer the notion of asymptotically scattered spaces to $G$-spaces and characterize asymptotically scattered spaces of groups.
\end{abstract}

\section{Introduction}

A metric space $(X,d)$ is called {\em thin} (or {\em asymptotically discrete}) if, for every $m\in\w=\{0,1,...\}$, there exists a bounded subset $V$ of $X$ such that $B(x,m)=\{x\}$ for each $x\in X\setminus V$. Here $B(x,m)=\{y\in X: d(x,y)\leq m\}$. A subset $V$ is bounded if $V \subseteq B(x_0,n)$, for some $x_0\in X$, $n\in \w$. We say that a subset $Y$ of $X$ is {\em thin} if the metric space $(Y,d|_Y)$ is thin. Clearly, each bounded subset of a metric space is thin, and each subset of a thin metric space is thin.

Following \cite{b1}, we say that a subset $Y$ of a metric space $(X,d)$ has {\em asymptotically isolated $m$-balls} if there exists a sequence $(y_n)_{n\in\w}$ in $Y$ and an increasing sequence $(m_n)_{n\in\w}$ in $\w$ such that $m_0>m$ and $B(x_n,m_n)\setminus B(x_n,m)=\varnothing$ for each $n\in\w$. If $Y$ has asymptotically isolated $m$-balls for some $m\in\w$, we say that $Y$ has {\em asymptotically isolated balls}.

The following definition was suggested by T. Banakh in a private conversation. A metric space $(X,d)$ is called {\em asymptotically scattered} if each unbounded subset of $X$ has asymptotically isolated balls. We say that a subset $Y$ of $X$ is asymptotically scattered if the metric space $(Y,d_Y)$ is asymptotically scattered. Clearly each thin metric space is asymptotically scattered, and each subset of an asymptotically scattered metric space is asymptotically scattered.

The  asymptotically scattered metric spaces can be considered as counterparts of scattered metric spaces. Recall that a topological space $X$ is scattered if each non-empty subset $Y$ of $X$ has a point isolated in $Y$.

Given two metric spaces $(X_1,d_1)$ and $(X_2,d_2)$, a bijection $f: X_1\to X_2$ is said to be an {\em asymorphism} if there are two sequences $(c_n)_{n\in\w}$ and $(c_n')_{n\in\w}$ in $\w$ such that, for each $n\in\w$ and $x,y\in X_1$,
$$d_1(x,y)\leq n\Rightarrow d_2(f(x),f(y))<c_n,$$
$$d_2(f(x),f(y))\leq n\Rightarrow d_2(x,y)<c_n'.$$
We note \cite[Theorem 2.1.1]{b9} that each metric space $(X,d)$ is asymorphic to a metric space $(X',d')$ such that $d'$ takes values in $\w$. In what follows all metrics under consideration are supposed to be integer valued.

A subset $L$ of a metric space $(X,d)$ is called {\it large} if there exists $m\in\w$ such that $B(L,m)=X$.

The metric spaces $(X_1,d_1)$, $(X_2,d_2)$ are called {\em coarsely equivalent} if there are large subsets $L_1\subseteq X_1$ and $L_2\subseteq X_2$ such that the metric spaces $(L_1,d_1)$ and $(L_2,d_2)$ are asymorphic.

We say that a property $\PP$ of metric spaces is {\em asymptotic} (resp. {\em coarse}) if $\PP$ is stable  under asymorphism (resp. coarse equivalence). It is easy to see that "thin" is an asymptotic but not coarse property, and "asymptotic scattered" is a coarse property.

A metric space $(X,d)$ is of {\em bounded geometry} if there exists $m\in\w$ and a function $c:\w\to\w$ such that an $m$-capacity of every ball $B(x,n)$ does not exceed $c(n)$. An {\em $m$-capacity} of a subset $Y$ of $X$ is supremum of cardinalities of $m$-discrete subsets of $Y$. A subset $Z$ is {\em $m$-discrete} if $d(x,y)>m$ for all distinct $x,y\in Z$.

A metric space $(X,d)$ is called {\em uniformly locally finite} if there is a function $c:\w\to\w$ such that $|B(x,n)|\leq c(n)$ for each $x\in X$ and $n\in\w$. It is easy to see \cite[Proposition 2]{b8} that $(X,d)$ is of bounded geometry if and only if $(X,d)$ is coarsely equivalent to some uniformly locally finite metric space.

Recall that a metric $d$ on a set $X$ is called an {\it ultrametric} if $d(x,y)\leq \max\{d(x,z),d(y,z)\}$ for all $x,y,z\in X$. By \cite[Theorem 3.11]{b2}, every ultrametric space of bounded geometry is coarsely equivalent to some subset of the Cantor macro-cube
$$2^{<\NN}=\{(x_i)_{i\in\NN}\in\{0,1\}^\NN:\exists n\in\NN\text{ }\forall m>n\text{ } x_m=0\}$$
endowed with the ultrametric
$$d((x_i)_{i\in\NN},(y_i)_{i\in\NN})=\max\{n\in\NN:x_n\neq y_n\}.$$
By \cite{b1}, an ultrametric space of bounded geometry without asymptotically isolated balls is coarsely equivalent to $2^{<\NN}$. However \cite{b7}, there are $\mathfrak{c}$ pairwise non-asymorphic uniformly locally finite ultrametric spaces.

In section~\ref{s2} we characterize the asymptotically scattered spaces in terms of prohibited subspaces. In section~\ref{s3} we classify the thin spaces up to coarse equivalence. In section~\ref{s4} we introduce the types of scattered spaces and construct the scattered spaces of distinct types. In sections~\ref{s5} and ~\ref{s6} we transfer the notion of asymptotically scattered metric spaces to $G$-spaces and characterize the asymptotically scattered subsets of a group.

\section{Characterizations}~\label{s2}
For a metric space $(X,d)$ and $m\in\w$, a sequence $x_1,...,x_n$ in $X$ is called an {\em $m$-chain} if $d(x_i,x_{i+1})\leq m$ for each $i\in\{1,...,n-1\}$. We set
$$B^\square(x,m)=\{y\in X:\text{ there is an $m$-chain between $x$ and $y$}\}$$
and say that $X$ is cellular if, for each $m\in\w$, there is $c_m\in\w$ such that $B^\square(x,m)\subseteq B(x,c_m)$ for every $x\in X$. By \cite[Theorems 3.1.1 and 3.1.3]{b9}, the following three conditions are equivalent: $(X,d)$ is cellular, $asdim(X)=0$, $(X,d)$ is asymorphic to some ultrametric space.

Let $(X,d)$ be an unbounded metric space. We define a function $h:X\times\omega\to\omega$ by
$$h(x,m)=\min\{n:B(x,n)\setminus B(x,m)\neq\varnothing\},$$
and note that $(X,d)$ is asymptotically scattered if and only if, for each unbounded subset $Y$ of $X$, there exist $m\in\w$ and an unbounded subset $Z\subseteq Y$ such that the set $\{h(x,m):x\in Z\}$ is infinite.
\begin{Th}~\label{t21} Every asymptotically scattered metric space $(X,d)$ is cellular. \end{Th}
\begin{proof}
If $(X,d)$ is not cellular then there exist $m\in\w$ and an unbounded injective sequence $(x_n)_{n\in\w}$ in $X$ such that $B^\square(x_n,m)\setminus B(x_n,n)\neq\varnothing$. It follows that $h(x_n,k)\leq m$ for each $k\in\{0,...,n\}$. We put $Y=\{x_n:n\in\w\}$ and observe that the set $\{h(x_n,k):n\in\w\}$ is finite for each $k\in\w$. Here $(X,d)$ is not asymptotically scattered.
\end{proof}
\begin{Th} If $X_1$, $X_2$ are asymptotically scattered subspaces of a metric space $(X,d)$ then $X_1\cup X_2$ is asymptotically scattered.\end{Th}
\begin{proof}
Let $Y$ be an unbounded subset of $X_1\cup X_2$. We may suppose that $Y\subseteq X_1$. Since $X_1$ is asymptotically scattered there exist $m\in\w$, a sequence $(y_n)_{n\in\w}$ in $Y$ and an increasing sequence $(m_n)_{n\in\w}$ in $\w$ such that for every $n\in\w$,
$$(X_1\cap B(y_n,m_n))\setminus B(y_n,m)=\varnothing.$$
If the set $\{n\in\w: B(y_n,m_n)\setminus B(y_n,m)\neq\varnothing\}$ is finite then $Y'=\{y_n:n\in\w\}$ has asymptotically isolated $m$-balls in $X_1\cup X_2$. Otherwise, we may suppose that $B(y_n,m_n)\setminus B(y_n,m)\neq\varnothing$ for each $n\in\w$. We pick $z_n\in B(y_n,m_n)\setminus B(y_n,m)$ such that
$$d(z_n,y_n)=\min\{d(z,y_n):z\in B(y_n,m_n)\setminus B(y_n,m)\}$$
If the sequence $(d(z_n,y_n)_{n\in\w})$ is unbounded in $\w$, $Y'$ has an asymptotically isolated $m$-balls in $X_1\cup X_2$. Otherwise, we take $t\in\w$ such that $d(z_n,y_n)\leq t$ for each $n\in\w$. Since $Z=\{z_n:n\in\w\}$ is an unbounded subset of $X_2$, $Z$ has an asymptotically isolated $m'$-balls in $X_2$ for some $m'\in\w$. Then $Y'$ has asymptotically isolated $(m+m'+t)$-balls in $X_1\cup X_2$.
\end{proof}
We define a metric $\rho'$ on $\w$ by the rule: $\rho'(n,n)=0$ and $\rho'(n,m)=\max\{n,m\}$ if $m\neq n$. Then we define a metric $\rho$ on $W=\w\times\w$ by the rule:
$$\rho((n,m),(n',m'))=\max\{\rho(n,n'),\rho(m,m')\}.$$
We consider two subspaces $W_1$ and $W_2$ of $(W,\rho)$
$$W_1=\{(n,0):n\in\w \},\text{ } W_2=W_1\cup\{(n,m):n>m>0\},$$
and observe that $W_1$ is isometric to the subspace $\{x\in 2^{<\NN}:supp(x)\leq 1\}$, $W_2$ is isometric to the subspace $\{x\in 2^{<\NN}:supp(x)\leq 2\}$ of the Cantor macro-cube $2^{<\NN}$. Here $supp(x)$ is the number of non-zero coordinates of $x=(x_n)_{n\in\NN}$.
\begin{Th}~\label{t23} A metric space $(X,d)$ is asymptotically scattered if and only if $(X,d)$ has no subspaces asymorphic to $W_2$.\end{Th}
\begin{proof} To show that $W_2$ is not asymptotically scattered, we use the function $h:W_2\times\w\to\w$. For $n>m\geq0$, we have  $h((n,0),m)=m+1$. Hence the subspace $W_1$ of $W_2$ has no asymptotically isolated balls.

Now assume that $X$ is not asymptotically scattered and find a subset $W'$ of $X$ asymorphic to $W_2$. We take a subset $X'=\{x_n:n\in\w\}$ without isolated balls. Passing to subsequences of $(x_n)_{n\in\w}$ $\w$ times, we can choose a sequence $(y_n)_{n\in\w}$ in $X'$ and an increasing sequence $(m_n)_{n\in\w}$ in $\w$ such that $m_0=0$ and

$(1)$ $(X\cap B(y_n,m_i))\setminus B(y_n,m_{i-1})\neq\varnothing$, $n>0$, $i\in\{1,...,n\};$

$(2)$ $B(y_n,m_n)\cap B(y_i,m_i)=\varnothing$, $0\leq i<n<\w$.

We use  $(1)$, $(2)$ to choose a subsequence $(z_{n0})_{n\in\w}$ of $(y_n)_{n\in\w}$, a subsequence $(k_n)_{n\in\w}$ and, for each $n>0$, the elements $z_{n1},...,z_{nn}$ such that

$(3)$ $z_{ni}\in(X\cap B(z_{n0},k_i))\setminus B(z_{n0},k_{i-1})$, $i\in\{1,...,n\};$

$(4)$ $d(z_{ni},z_{nj})>j$, $0\leq i<j\leq n;$

$(5)$ $B(z_{n0},k_n)\cap B(z_{i0},k_i)=\varnothing$, $0\leq i<n<\w$.

Then we consider a set $W'=\{z_{00}\}\cup\{z_{ni}: n > i\geq0\}$ and define a mapping $f:W_2\to W'$ by the rule: $f(0,0)=z_{00}$ and $f(n,i)=z_{ni}$, $n>i\geq0$. By $(3)$ and $(4)$, $f$ is a bijection.

If $f$ is not an asymorphism, we get the following two cases.

{\em Case $1$.} There exist $t>0$ and two sequences $(a_n)_{n\in\w}$, $(b_n)_{n\in\w}$ in $W_2$ such that $\rho(a_n,b_n)=t$ but $d(f(a_n),f(b_n))\to\infty$. We may suppose that $a_n=(c_n,t)$, $b_n=(c_n,s)$, $s<t$. But $d(z_{c_nt},z_{c_ns})\leq2k_t$, a contradiction.

{\em Case $2$.} There exist $t>0$ and two sequences $(a_n)_{n\in\w}$, $(b_n)_{n\in\w}$ in $W_2$ such that $d(f(a_n),f(b_n))=t$ but  $\rho(a_n,b_n)\to\infty$. In view of $(4)$, $(5)$, we may suppose that $f(a_n)=z_{c_ni_n}$, $f(b_n)=z_{c_nj_n}$ and $0\leq i_n<j_n<t$. But then $\rho(a_n,b_n)<t$, a contradiction.
\end{proof}
\section{Thin spaces} ~\label{s3}
The following three theorems are from \cite{b5}.
\begin{Th} A metric space $X$ is thin if and only if each unbounded subset of $X$ has asymptotically isolated $0$-balls. \end{Th}

We say that a metric space $X$ is {\em coarsely thin} if $X$ is coarsely equivalent to some thin space.
\begin{Th}~\label{t32} For a metric space $X$, the following statements are equivalent

$(i)$ $X$ is coarsely thin;

$(ii)$ $X$ contains large thin subset;

$(iii)$ there exists $m\in\w$ such that each unbounded subset of $X$ has asymptotically isolated $m$-balls. \end{Th}

A subset $Y$ of a metric space $X$ is called {\em asymptotically isolated} if, for each $m\in\w$, there is a bounded subset $V$ of $X$ such that $B(y,m)\subseteq Y$ for each $y\in Y\setminus V$.
\begin{Th} A metric space $X$ is asymptotically scattered if and only if each unbounded subset of $X$ has an asymptotically isolated coarsely thin subset. \end{Th}

To classify the thin metric spaces up to coarse equivalence, we take an infinite cardinal $\varkappa$, partition $\varkappa$ into countable many subsets $\varkappa=\bigcup_{n\in\NN} K_n$, $|K_n|=\varkappa_n$, define a metric $\rho$ on $\kappa$ by $\rho(x,x)=0$ and $\rho(x,y)=\max\{n,m\}$ if $x\in K_n$, $y\in K_m$ and $x\neq y$.

We denote the obtained metric space by $T(\varkappa_n)_{n\in\NN}$ and observe that $T(1,1,...)$ is isometric to $W_1$.
\begin{Lm}~\label{l34} Let $(X,d)$ be an unbounded thin ultrametric space of cardinality $\varkappa$. Then there is a partition $\varkappa=\bigcup_{n\in\NN} K_n$, $|K_n|=\varkappa_n$ such that $(X,d)$ is asymorphic to $T(\varkappa_n)_{n\in\NN}$. \end{Lm}
\begin{proof}
We partition $X=\bigcup_{i\in\w}X_{i1}$ by the equivalence $\sim_1$ defined by $x\sim_1 y\Leftrightarrow x=y\vee d(x,y)=1$. Since $X$ is thin, the set $X_1=\bigcup\{X_{i1}:|X_{i1}|>1\}$ is bounded. We partition $X\setminus X_1=\bigcup_{i\in\omega}X_{i2}$ by the equivalence $\sim_2$ defined by $x\sim_2 y\Leftrightarrow x=y\vee d(x,y)=2$. Since $X\setminus X_1$ is thin, the set $X_2=\cup\{X_{i2}:|X_{i2}|>1\}$ is bounded. We partition $X\setminus(X_1\cup X_2)$ by the equivalence $\sim_3$ defined by $x\sim_3 y\Leftrightarrow x=y\vee d(x,y)=3$, and so on.

After $\omega$ steps, we get a partition $X=\bigcup_{n\in\NN}X_n$, $|X_n|=\varkappa_n$. Then we partition $\varkappa=\bigcup K_n$ so that $|K_n|=\varkappa_n$. For each $n\in\NN$, let $f_n:X_n\to K_n$ be an arbitrary bijection. It is easy to see that $f=\bigcup_{n\in\NN}f_n$ is an asymorphism between $X$ and $T(\kappa_n)_{n\in\NN}$.
\end{proof}

For a metric space $X$, the minimal cardinality $asden(X)$ of large subsets of $X$ is called an {\em asymptotic density} of $X$. Clearly $asden(X)=1$ if and only if $X$ is bounded. We note also that asymptotic density is invariant under coarse equivalence, and each metric space is coarsely equivalent to a metric space $X$ such that $|X|=asden(X)$.
\begin{Th}~\label{t35} Let $X$ be an unbounded thin metric space such that $|X|=asden(X)=\varkappa$. Then the following statements hold

$(i)$ if $\varkappa=\aleph_0$ then $X$ is coarsely equivalent either to $T(1,1,...)$ or to $T(\aleph_0,\aleph_0,...)$;

$(ii)$ if $cf(\varkappa)>\aleph_0$ then $X$ is coarsely equivalent to $T(\varkappa,\varkappa,...)$;

$(iii)$ if $\varkappa>\aleph_0$, $cf(\varkappa)=\aleph_0$ and $(\varkappa_n')_{n\in\NN}$ is a sequence of cardinals such that $\varkappa_n'<\varkappa_{n+1}'$ and $\sup\{\varkappa_n':n\in\NN\}=\varkappa$ then $X$ is coarsely equivalent either to $T(\varkappa_n')_{n\in\NN}$ or to $T(\varkappa,\varkappa,...)$. \end{Th}
\begin{proof}
In view of Theorem~\ref{t21}, we may suppose that $X$ is ultrametric. By Lemma~\ref{l34}, there is a partition $\varkappa=\bigcup_{n\in\NN} K_n$, $|K_n|=\varkappa_n$ such that $X$ is asymorphic to $T(\kappa_n)_{n\in\NN}$.

$(i)$ We consider two cases.

{\em Case $1$}. There exists $n_0\in\NN$ such that $\kappa_n<\aleph_0$ for each $n>n_0$. The subset $\aleph_0\setminus\bigcup_{i\leq n_0}K_i$ is large in $T(\kappa_n)_{n\in\NN}$ and hence coarsely equivalence to $T(\varkappa_n)_{n\in\NN}$. We take an arbitrary bijection $f:\aleph_0\setminus\bigcup_{i\leq n_0}K_i\to T(1,1,...)$ and note that $f$ is an asymorphism.

{\em Case $2$}. There exists an increasing sequence $(n_k)_{k\in\NN}$ in $\NN$ such that $\varkappa_{n_k}=\aleph_0$ for each $k\in\NN$. We partition $\aleph_0=\bigcup_{k\in\NN}K_k'$ so that $|K_k'|=\aleph_0$ for each $k\in\NN$. Then we take an arbitrary bijection $f:\aleph_0\to\aleph_0$ such that $f(\bigcup_{i\leq n_1}K_i)=K_1'$, $f(\bigcup_{n_1<i\leq n_2}K_i)=K_2,...$. It is easy to see that $f$ is an asymorphism between $T(\kappa_n)_{n\in\NN}$ and $T(\aleph_0,\aleph_0,...)$.

$(ii)$ Asume that there exists $n_0\in\NN$ such that $\varkappa_n<\varkappa$ for each $n>n_0$. On one hand, the subspace $\varkappa\setminus\bigcup_{i\leq n_0}K_i$ is coarsely equivalent to $T(\varkappa_n)_{n\in\NN}$. On the other hand, $|\varkappa\setminus\bigcup_{i\leq n_0} K_i|<\varkappa$ because $cf(\varkappa)>\aleph_0$. Thus $asden(X)<\varkappa$ contradicting the assumption.

Hence there exists an increasing sequence $(n_k)_{k\in\NN}$ in $\NN$ such that $K_{n_k}=\varkappa$ for each $k\in\NN$. We partition $\varkappa=\bigcup_{k\in\NN}K_k'$, $|K_k'|=\varkappa$ and define an asymorphism $f:T(\kappa_n)_{n\in\NN}\to T(\kappa,\kappa,...)$ as in the Case $2$ of $(i)$.

$(iii)$ We consider two cases.

{\em Case $1$}. There is $m\in\NN$ such that $\varkappa_n<\varkappa$ for each $n\geq m$. We may suppose that $m=1$. We partition $\varkappa=\bigcup_{n\in\NN}K_n'$ such that $K_n'=\kappa_n'$ and choose two increasing sequences $(n_k)_{k\in\NN}$, $(m_k)_{k\in\NN}$ in $\NN$ such that $\varkappa_{m_1}'\leq|\bigcup_{i\leq n_1}K_i|<\varkappa_{m_2}'$ and, for each $k>1$,
$$\varkappa_{m_k}'\leq|\bigcup_{n_k<i\leq n_{k+1}}K_i|<\varkappa_{m_{k+1}}'.$$
Then we choose a bijection $f:\varkappa\to\varkappa$ such that for each $k\in\NN$,
$$\bigcup_{i\leq m_k}K_i'\subseteq f(\bigcup_{i\leq n_k}K_i)\subseteq\bigcup_{i\leq m_{k+1}}K_i'.$$
Then $f$ is a desired asymorphism between $T(\varkappa_n)_{n\in\NN}$ and $T(\varkappa_n')_{n\in\NN}$.

{\em Case $2$}. There is an increasing sequence $(n_k)_{k\in\NN}$ in $\NN$ such that $\varkappa_{n_k}=\varkappa$ for each $k\in\NN$. Then an asymorphism $f:T(\varkappa_n)_{n\in\NN}\to T(\varkappa,\varkappa,...)$ can be defined as in the Case $2$ of $(i)$.
\end{proof}
\begin{Rm} The metric spaces $T(1,1,...)$ and $T(\aleph_0,\aleph_0,...)$ are not coarsely equivalent because $T(1,1,...)$ is uniformly locally finite but $T(\aleph_0,\aleph_0,...)$ is not of bounded geometry. We show that the metric spaces $T(\varkappa,\varkappa,...)$ and $T(\varkappa_n')_{n\in\NN}$ from $(iii)$ of Theorem~\ref{t35} are not coarsely equivalent. We note that each large subset of $T(\varkappa,\varkappa,...)$ (resp. $T(\varkappa_n')_{n\in\NN}$) is asymorphic to $T(\varkappa,\varkappa,...)$ (resp. $T(\varkappa_n')_{n\in\NN}$). Let $f:T(\varkappa,\varkappa,...)\to T(\varkappa_n')_{n\in\NN}$ is a bijection, and let $\varkappa=\bigcup_{n\in\NN}K_n$ be a partition which determine $T(\varkappa,\varkappa,...)$. On one hand $K_1$ is bounded in $T(\varkappa,\varkappa,...)$. On the other hand $|f(K_1)|=\varkappa$ so $f(K_1)$ is unbounded in $T(\varkappa_n')_{n\in\NN}$. Hence $f$ is not an asymorphism. \end{Rm}

\section{Scattered types}~\label{s4}
Let $X$ be a scattered metric space. We say that $X$ is of {\em type $0$} if $X$ is bounded, and $X$ is of {\em type $1$} if $X$ is unbounded and coarsely thin. For $m>1$, $X$ is of {\em type $m$} if and only if $X$ can be partitioned in $m$ coarsely thin subsets, but $X$ is not of type less then $m$. If $X$ is not of type $m$ for each $m\in\w$, we say that $X$ is of {\em infinite type}. Clearly, the types are invariant under coarse equivalence.

For $m\in\w$, we say that a metric space $X$ is {\em $m$-thin} if, for every $n\in\w$, there exists a bounded subset $V$ of $X$ such that $|B(x,n)|\leq m$ for each $x\in X\setminus V$. Clearly, $X$ is $0$-thin if and only if $X$ is bounded. By \cite{b4}, every unbounded $m$-thin space can be partitioned in $\leq m$ thin subsets. Applying Theorem~\ref{t32}, we conclude that an unbounded metric space $X$ is of type $m$ if and only if $m$ is the minimal number such that $X$ contains a large $m$-thin subsets. Thus, a classification of scattered metric spaces of finite types is reduced to the case of $m$-thin spaces.

Now we consider some construction of asymptotically scattered spaces. Let $(X_n)_{n\in\w}$ be a sequence of subsets of $\w$ such that $\min X_n>n$ for each $n>0$. We denote by $W(X_n)_{n\in\w}$ the subspace $\bigcup\{X_n\times\{n\}:n\in\w\}$ of $W_2$. Applying Theorem~\ref{t23}, we conclude that $W(X_n)_{n\in\w}$ is asymptotically scattered if and only if, for each infinite subset $I$ of $\w$, there exists a finite subset $F\subset I$ such that $\bigcap_{n\in F}X_n$ is finite.

By Theorem~\ref{t35} $(i)$, each metric space of bounded geometry of type $1$ is coarsely equivalent to $W_1$. In the following example we show that, up to coarse equivalence, there is only one metric space of bounded geometry of type $2$.
\begin{Ex} We take a partition $\{X_n:n\in\w\}$ of $\w$ into infinite subsets such that $\min X_n>n$ for each $n>0$, and show that each metric space $X$ of bounded geometry of type $2$ is coarsely equivalent to $W(X_n)_{n\in\w}$. We may suppose that $X$ is subspace of $2^{<\NN}$ and $X=Y\cup Z$, $Y\cap Z=\varnothing$ where $Y,Z$ are unbounded thin subsets of $2^{<\NN}$. For each $z\in Z$, we pick $a(z)\in 2^{<\NN}$ such that $z+a(z)\in Y$ and $\parallel a(z)\parallel=\min\{d(z,y):y\in Y\}$. Let $A=\{a(z):z\in Z\}$, and for each $a\in A$, $Z_a=\{z\in Z: a(z)=a\}$. We set $A_0=\{a\in A:Z_a\text{ is infinite}\}$, $A_1=A\setminus A_0$. If $A_0$ is finite then $Y\cup Z$ is coarsely thin, so we assume that $A_0$ is infinite, $A_0=\{a_n:n>0\}$. Since $Y\cup Z$ is $2$-thin, $Z_a\cap Z_b$ is finite for all distinct $a,b\in A_0$. We put $Z_1=Z_{a_1}$ and $Z_n=Z_{a_n}\setminus(Z_{a_0}\cup...\cup Z_{a_{n-1}})$ for all $n>1$. Then we define a bijection $f:W(X_n)_{n\in\w}\to Y\cup Z$ such that $f(X_n\times\{n\})=Z_n$ for each $n>0$, so $f(X_0\times\{0\})=Y\cup\bigcup\{Z_a: a\in A_1\}$. This $f$ is a desired asymorphism between $W(X_n)_{n\in\w}$ and $X$.\end{Ex}
\begin{Ex} For each $n\geq2$, we construct a subset of $W_2$ of type $n$. To this end, we take a sequence $(X_n)_{n\in\w}$ of infinite subsets of $\w$ such that $\min X_n>n$, $n>0$ and

$(1)$ $\bigcap_{n\in F} X_n$ is finite for each $F\subset\w$, $|F|=n$;

$(2)$ $\bigcap_{n\in H} X_n$ is infinite for each $H\subset\w$, $|F|=n-1$.

By $(1)$, $W(X_n)_{n\in\w}$ is $n$-thin. By $(2)$, each large subset of $W(X_n)_{n\in\w}$ is not $(n-1)$-thin. Hence, $W(X_n)_{n\in\w}$ is of type $n$.\end{Ex}
\begin{Pb} For each $n\geq3$, classify the metric spaces of bounded geometry of type $n$ up to coarse equivalence. \end{Pb}

\section{Ballean context} ~\label{s5}
Following \cite{b9}, we say that a {\em ball structure} is a triple $\BB=(X,P,B)$, where $X$, $P$ are non-empty sets and, for every $x\in X$
and $\alpha\in P$, $B(x,\alpha)$ is a subset of $X$ which is called a {\em ball of radius $\alpha$ around $x$}. It is supposed that $x\in B(x,\alpha)$ for all $x\in X$ and $\alpha\in P$. The set $X$ is called the {\em support} of $\BB$, $P$ is called the set of {\em radii}.

Given any $x\in X$, $A\subseteq X$, $\alpha\in P$, we set
$$B^\ast(x.\alpha)=\{y\in X: x\in B(y,\alpha)\},\text{ } B(A,\alpha)=\bigcup_{a\in A}B(a,\alpha).$$
A ball structure $\BB=(X,P,B)$ is called a {\em ballean} if

$\bullet$ for any $\alpha,\beta\in P$, there exist $\alpha',\beta'\in P$ such that, for every $x\in X$,
$$B(x,\alpha)\subseteq B^\ast(x,\alpha'),\text{ }B^\ast(x,\beta)\subseteq B(x,\beta');$$

$\bullet$ for any $\alpha,\beta\in P$, there exists $\gamma\in P$ such that, for every $x\in X$,
$$B(B(x,\alpha),\beta)\subseteq B(x,\gamma).$$
A ballean $\BB$ on $X$ can also be determined in terms of entourages of diagonal $\Delta_X$ of $X\times X$, in this case it is called a {\em coarse structure} \cite{b10}. For our "scattered" goal, we prefer the ball language.

We suppose that all ballean under consideration are {\em connected}, i.e. for any $x,y\in X$ there exists $\alpha\in P$ such that $y\in B(x,\alpha)$.

Let $\BB_1=(X_1,P_1,B_1)$ and $\BB_2=(X_2,P_2,B_2)$ be balleans. A mapping $f:X_1\to X_2$ is called a $\prec$-mapping if, for every $\alpha\in P_1$, there exists $\beta\in P_2$ such that, for every $x\in X$ $f(B_1(x,\alpha))\subseteq B_2(f(x),\beta)$. If there exists a bijection $f: X_1\to X_2$ such that $f$ and $f^{-1}$ are $\prec$-mappings, $\BB_1$ and $\BB_2$ are called {\em asymorphic}.

For a ballean $\BB=(X,P,B)$, a subset $Y\subseteq X$ is called {\em large} if there is $\alpha\in P$ such that $X=B(Y,\alpha)$. A subset $V\subseteq X$ is called bounded if $V\subseteq B(x,\alpha)$ for some $x\in X$ and $\alpha\in P$. Each non-empty subset $Y\subseteq X$ defines a {\em subballean} $\BB_Y=(Y,P,B_Y)$, where $B_Y(Y,\alpha)=Y\cap B(y,\alpha)$.

By definition, two balleans $\BB=(X,P,B)$ and $\BB'=(X',P',B')$ are {\em coarsely equivalent} if there exist large subsets $Y\subseteq X$, $Y'\subseteq X'$ such that the subballeans $\BB_Y$ and $\BB_{Y'}$ are asymorphic.

Let $G$ be a group, $\II$ be an ideal in the Boolean algebra $\PP_G$ of all subsets of $G$, i.e. $\varnothing\in\II$ and if $A,B\in\II$ and $A'\subseteq A$, then $A\cup B\in\II$ and $A'\in\II$. An ideal $\II$ is called a {\em group ideal} if $\II$ contains all finite subsets of $G$ and, for all $A,B\in\II$, we have $AB\in\II$ and $A^{-1}\in\II$.

Now let $X$ be a transitive $G$-space with the action $G\times X\to X$, $(g,x)\mapsto gx$, and let $\II$ be a group ideal in $\PP_G$. We define a ballean $\BB(G,X,\II)$ as a triple $(X,\II,B)$, where $B(x,A)=Ax\cup\{x\}$ for all $x\in X$, $A\in\II$. By \cite[Theorem 1]{b6} every ballean $\BB$ with the support $X$ is asymorphic to the ballean $\BB(G,X,\II)$ for some group $G$ of permutation of $X$ and some group ideal $\II$ in $\PP_G$.

For a group $G$, we denote by $\mathfrak{F}_G$ the ideal of all finite subsets of $G$. Each metric space $(X,d)$ can be considered as the ballean $(X,\w,B_d)$.

The following two theorems are from \cite[Theorem 2.1.1]{b9} and \cite[Theorem 1]{b8}.
\begin{Th} Let $G$ be a countable transitive group of permutations of a set $X$. Then there exists a uniformly locally finite metric $d$ on $X$ such that the ballean $\BB(G,\mathfrak{F}_G,X)$ is asymorphic to $(X,d)$.\end{Th}
\begin{Th} Let $(X,d)$ be a uniformly locally finite metric space. Then there exists a countable group $G$ of permutations of $X$ such that $(X,d)$ is asymorphic to $\BB(G,\mathfrak{F}_G,X)$. \end{Th}

We say that a ballean $\BB=(X,P,B)$ is {\em thin} if, for every $\alpha\in P$, there exists a bounded subset $V$ of $X$ such that $B(x,\alpha)=\{x\}$ for each $x\in X\setminus V$. A subset $Y\subseteq X$ is called thin if the subballean $\BB_Y$ is thin.

We use the natural preordering $<$ on $P$: $\alpha<\beta$ if and only if $B(x,\alpha)\subseteq B(x,\beta)$ for each $x\in X$. We say that a subset $Y\subseteq X$ has {\em asymptotically $\alpha$-isolated balls} if, for every $\beta>\alpha$, there exists $y\in Y$ such that $B(y,\beta)\setminus B(y,\alpha)=\varnothing$. If $Y$ has asymptotically $\alpha$-isolated balls for some $\alpha\in P$, we say that $Y$ has {\em asymptotically isolated balls}.

A ballean $\BB=(X,P,B)$ is called {\em asymptotically scattered} if each unbounded subset of $X$ has asymptotically isolated balls. A subset $Y\subseteq X$ is called {\em asymptotically scattered} if the subballean $\BB_Y$ is asymptotically scattered.

\section{Scattered subsets of a group} ~\label{s6}
In this section, we consider a group $G$ as a $\BB(G,\mathfrak{F}_G,X)$ where $X=G$ and $G$ acts on $X$ by left translation. We remind that, for $g\in G$ and $F\in\mathfrak{F}_G$, $$B(g,F)=Fg\cup\{g\}.$$

A subset $A$ of a group $G$ is called {\em sparse} \cite{b3}, if for every infinite subset $Y$ of $G$, there exists a finite subset $F\subset Y$ such that $\bigcap_{g\in F}gA$ is finite.
\begin{Th} For a subset $A$ of a countable group $G$, the following statements are equivalent

$(i)$ $A$ is asymptotically scattered;

$(ii)$ $A$ is sparse;

$(iii)$ for each free ultrafilter $\UU$ on $G$ with $A\in\UU$, the set $\{g\in G:A\in g\UU\}$ is finite, where $g\UU=\{gU:U\in\UU\}$.
\end{Th}
\begin{proof}
$(i)\Rightarrow(iii)$. Assume that there exists an infinite subset $\{g_n:n\in\w\}$ of $G$ such that $A\in g_n\UU$ for each $n\in\w$. We choose a decreasing sequence $U_n$ of members of $\UU$ such that $g_nU_n\subseteq A$, $U_n\subseteq A$ and pick $x_n\in U_n\setminus U_{n-1}$. Then $g_n\{x_n,x_{n+1},...\}\subseteq A$ for each $n\in\w$. It follows that the set $\{x_n:n>0\}$ has no isolated balls in $A$, so $A$ is not asymptotically scattered.

$(iii)\Rightarrow(ii)$. Suppose that $A$ is not sparse and choose a subset $Y=\{y_n:n\in\w\}$ of $G$ such that $\bigcap_{0\leq i\leq n} y_iA$ is infinite for each $n\in\w$. We take an arbitrary free ultrafilter $\VV$ on $G$ such that $\{y_nA:n\in\w\}\subset\VV$, put $\UU=y_0^{-1}\VV$ and note that $A\in\UU$ and $A\in y_n^{-1}y_0\UU$ for each $n>1$, so $(iii)$ does not hold.

$(ii)\Rightarrow(i)$. If $A$ is not asymptotically scattered, we can choose an injective sequence $(x_n)_{n\in\w}$ in $A$ and an injective sequence $(y_n)_{n\in\w}$ in $G$ such that $y_n\{x_n.x_{n+1},...\}\subseteq A$ for each $n\in\w$. Then $\bigcap_{0\leq i\leq n} y_i^{-1}A$ is infinite for each $n\in\w$, so $A$ is not scattered.
\end{proof}
\begin{Th} Every countable group $G$ contains an asymptotically scattered subset of infinite type and, for each $m\in\w$, an asymptotically scattered subset of type $m$.\end{Th}
\begin{proof}
We write $G$ as a union $G=\bigcup_{n\in\w}K_n$, $K_0=\{e\}$ of an increasing chain of finite symmetric subsets.

In the proof of Theorem $2.1$ from \cite{b3}, we constructed a sparse subset $A$ of the form $A=\bigcup_{n\in\w} F_nx_n$ where $(F_n)_{n\in\w}$ is an appropriate sequence in $\mathfrak{F}_G$. By the construction, $A$ has the following property

$(1)$ for all $n>0$ and $m>0$, there exists $F\in\mathfrak{F}_G$ and an infinite subset $I\subset\w$ such that $|F|=m+1$, $F=F_n$ for each $n\in I$ and $K_nxx_i\cap A=\{xx_i\}$ for all $x\in F$ and $i\in I$.

We fix $m>0$, take an arbitrary large subset $L\subseteq A$ and pick $n\in\w$ such that $A\subseteq K_nL$. By $(1)$, $F_ix_i\subset L$ for each $i\in I$. Since $F=F_i$ and $|F|=m+1$, we conclude that $L$ is not $m$-thin. Hence $A$ is of infinite type.

In the proof of Theorem 1.1 \cite{b3}, for each $m>0$, we constructed a sparse subset $A=\bigcup_{n\in\w}F_nx_n$ with following properties

$(2)$ $|F_i|=m$, $K_iF_ix_i\cap K_nF_nx_n=\varnothing$, $0\leq i<n<\w$;

$(3)$ for each $n\in\w$, there exist $F\in\mathfrak{F}_G$ and an infinite subset $I\subset\w$ such that $F=F_i$ and $K_nxx_i\cap A=\{xx_i\}$ for all $i\in I$ and $x\in F$.

By $(2)$, $A$ is $m$-thin. Let $L$ be a large subset of $A$. We pick $n\in\w$ such that $A\subseteq K_nL$. By $(3)$, $F_ix_i\subset L$ for each $i\in I$. Since $F=F_i$ and $|F|=m$, we conclude that $L$ is not $(m-1)$-thin. Hence $A$ is of type $m$.
\end{proof}

Department of Cybernetics

Kyiv University

Volodimirska 64

01033 Kyiv

Ukraine

i.v.protasov@gmail.com 
\end{document}